\documentclass[11pt]{article}
\usepackage{graphicx}
\usepackage{amsmath}
\usepackage{amsthm}

\newfont{\bb}{msbm10 at 10pt}
\def\r{\hbox{\bb R}}
\def\s{\hbox{\bb S}}
\def\z{\hbox{\bb Z}}
\def\h{\hbox{\bb H}}

\newtheorem{theorem}{Theorem}[section]

\newtheorem{definition}[theorem]{Definition}

\newtheorem{remark}[theorem]{Remark}

\begin{document}
\title{Invariant surfaces in  homogenous space Sol \\with constant curvature}
\author{Rafael L\'opez\footnote{Partially supported by MEC-FEDER
 grant no. MTM2007-61775 and
Junta de Andaluc\'{\i}a grant no. P06-FQM-01642.}\\Departamento de Geometr\'{\i}a y Topolog\'{\i}a\\ Universidad de Granada\\ 18071 Granada, Spain\\
email: rcamino@ugr.es}

\date{}
\maketitle

\begin{abstract} A surface in  homogenous space Sol is said to be an  invariant surface if it is invariant under some of the two 1-parameter groups of isometries of the ambient space whose fix point sets are totally geodesic surfaces.  In this work we study invariant surfaces that satisfy a certain condition on their curvatures. We classify invariant surfaces with constant mean curvature and constant  Gaussian curvature. Also, we characterize  invariant surfaces that satisfy a linear Weingarten relation.
\end{abstract}



%%%%%%%%%%%%%%%%%%%%%%%%%%%%%%%%%%%%%%%%%%%%%%%%%%%%%%%%%%%%%%%
\section{Introduction}
%%%%%%%%%%%%%%%%%%%%%%%%%%%%%%%%%%%%%%%%%%%%%%%%%%%%%%%%%%%%%%%

The space Sol is a  simply connected  homogeneous 3-manifold whose isometry group has dimension $3$
and it is one of the eight models of geometry of Thurston \cite{th}.  As Riemannian manifold, the space Sol  can be
represented by $\r^3$ equipped with the metric
$$\langle,\rangle=ds^2=e^{2z}dx^2+e^{-2z}dy^2+dz^2$$
 where $(x,y,z)$ are canonical coordinates of $\r^3$. The space Sol, with the group operation
 $$(x,y,z)\ast (x',y',z')=(x+e^{-z}x',y+e^{z}y',z+z'),$$
is a Lie group and the metric $ds^2$ is  left-invariant.

Although Sol is a homogenous space and the action of the isometry group is transitive, the fact that the number of isometries is low (for example, there are no rotations) makes that the knowledge of the geometry of submanifolds is far to be complete. For example, it is known the geodesics (\cite{tr}) and more recently, the totally umbilical surfaces (\cite{sot}) and some properties on  surfaces with constant mean curvature (\cite{dm,il,lo}).

A first step in the understanding of the geometry of Sol consists into consider some type of symmetric property in the surface that allows easier to realize a problem of classification. On the other hand, it is also natural to take some assumptions of constancy on the curvatures of the surface. For these reasons, we will study surfaces invariant by a group of isometries of the ambient space as well as that some condition on the curvature, for example, that  the mean curvature or the Gaussian curvature is constant.

As we have pointed out, the isometry group Iso(Sol) has dimension $3$ and the component of the identity is generated by the following families of isometries:
\begin{eqnarray*}
& & T_{1,c}(x,y,z):=(x+c,y,z)\\
& &T_{2,c}(x,y,z):=(x,y+c,z)\\
& &T_{3,c}(x,y,z):=(e^{-c}x,e^{c}y,z+c),
\end{eqnarray*}
where $c\in\r$ is a real parameter. These isometries are left multiplications by elements of Sol and so, they are left-translations with respect to the structure of Lie group. Remark that the elements $T_{1,c}$ and $T_{2,c}$ are Euclidean translations along horizontal vector and that the set of fixed points are totally geodesic surfaces in
Sol. In this work we consider surfaces invariant under the 1-parametric group of isometries $T_{i,c}$, with $i=1,2$.

\begin{definition} A surface $S$ in Sol is said an invariant surface if it is invariant under one of the 1-parameter groups of isometries $\{T_{i,c};c\in\r\}$, with $i=1,2$.
\end{definition}

After an isometry of the ambient space, an invariant surface under the group $\{T_{2,c}\}_{c\in\r}$ converts into an invariant surface under the group $\{T_{1,c}\}_{c\in\r}$: this can done by
  taking the isometry of Sol given by $\phi(x,y,z)=(y,x,-z)$. Thus, throughout this work, we consider invariant surfaces under the first group $\{T_{1,c}\}_{c\in\r}$.

In this paper we study surfaces in Sol with some condition on their curvatures. In Section \ref{sect3} we classify all invariant surfaces of Sol with constant mean curvature $H$, including minimal surfaces (some pictures of surfaces with  $H\not=0$ appeared in \cite{dm}). In Section \ref{sect4} we construct and classify all invariant surfaces with constant (intrinsic or extrinsic) Gaussian curvature  ($K_{int}$ and $K_{ext}$). The fact that the Sol geometry has not constant sectional curvature makes that the constancy of $K_{int}$ or  $K_{ext}$ does not imply the other one. Finally in Section \ref{sect5} we initiate the study of linear Weingarten invariant surfaces by considering a relation of type $\kappa_1=m\kappa_2$, where $\kappa_i$ are the principal curvatures and $m\in\r$.

The study of surfaces with constant curvature, specially with constant mean curvature, in homogeneous 3-spaces and invariant under the action of a one-parameter group of isometries of the ambient space has been recently of interest for many geometers. Examples  can seen when the  ambient space is the Heisenberg group (\cite{cpr,fmp,in,mo2})  and the product space $\h^2\times\r$ (\cite{mo1,on,st}).

%%%%%%%%%%%%%%%%%%%%%%%%%%%%%%%%%%%%%%%%%%%%%%%%%%%%%%%%%%%%%%%%%%
\section{Local computations of curvatures}
%%%%%%%%%%%%%%%%%%%%%%%%%%%%%%%%%%%%%%%%%%%%%%%%%%%%%%%%%%%%%%%%%%

In this section we will recall some basic geometric properties of the space Sol and we will compute the curvatures of an invariant surface. See  also \cite{dm,tr}. With respect to the metric $ds^2$ an orthonormal basis of left-invariant vector fields is given by
$$E_1=e^{-z}\frac{\partial}{\partial x},\ \ E_2=e^{z}\frac{\partial}{\partial y},\ \ E_3=\frac{\partial}{\partial z}.$$
It is well known that the isometry group of Sol has dimension three. In particular, we can choose the following basis of Killing vector fields:
$$\frac{\partial}{\partial x},\hspace*{.5cm} \frac{\partial}{\partial y},\hspace*{.5cm}  -x\frac{\partial}{\partial x}+y\frac{\partial}{\partial y}+\frac{\partial}{\partial z}.$$
The one-parameter subgroups of isometries generated by the above three Killing vector fields were described in Introduction, namely, $\{T_{i,c};c\in\r\}$, $1\leq i\leq 3$, respectively.

The understanding of the geometry of Sol is given by the next three foliations:
\begin{eqnarray*}
& &{\cal F}_1: \{P_t=\{(t,y,z);y,z\in\r\}\}_{t\in\r}\\
& &{\cal F}_2: \{Q_t=\{(x,t,z);x,z\in\r\}\}_{t\in\r}\\
& &{\cal F}_3: \{R_t=\{(x,y,t);x,y\in\r\}\}_{t\in\r}.
\end{eqnarray*}
The foliations ${\cal F}_1$ and ${\cal F}_2$  are determined by the isometry groups $\{T_{1,c}\}_{c\in\r}$ and $\{T_{2,c}\}_{c\in\r}$ respectively, and they
describe (the only) totally geodesic surfaces of Sol, being  each leaf isometric to a hyperbolic plane;  the foliation ${\cal F}_3$ realizes by minimal surfaces, all them isometric to  Euclidean plane.

The Riemannian connection ${\nabla}$ of Sol with respect to $\{E_1,E_2,E_3\}$  is
$$\begin{array}{lll}
{\nabla}_{E_1} E_1=-E_3 & {\nabla}_{E_1}E_2=0&{\nabla}_{E_1}E_3=E_1\\
{\nabla}_{E_2} E_1=0 & {\nabla}_{E_2}E_2=E_3&{\nabla}_{E_2}E_3=-E_2\\
{\nabla}_{E_3} E_1=0 & {\nabla}_{E_3}E_2=0&{\nabla}_{E_3}E_3=0\\
\end{array}
$$
Now we are going to compute the curvatures of an invariant surface. An invariant surface $S$ is determined by the  intersection curve $\alpha$ with each one of the leaves of the corresponding foliation together the group of isometries. Any curve $\alpha$ is called a generating curve of the surface, and by our choice of group, we will consider that $\alpha$ is the curve $S\cap P_0$. Let us take a parametrization of $\alpha$ given by
$\alpha(s)=(0,y(s),z(s))$, $s\in I$,  where $s$ is the arc-length parameter. Thus
$$e^{-z(s)}y'(s)=\cos\theta(s),\hspace*{1cm}z'(s)=\sin\theta(s),$$
where $\theta=\theta(s)$ is a certain smooth function. We parametrize $S$ by
$$X(s,t)=(t,y(s),z(s)),\ \ s\in I\subset\r, \ t\in\r.$$
From now, we drop the dependence of $y$, $z$ and $\theta$ on the variable $s$. We have
\begin{eqnarray*}
& &e_1:=X_s=(0,y',z')=\cos\theta E_2+\sin\theta E_3.\\
& &e_2:=X_t=(1,0,0)=e^{z}E_1.
\end{eqnarray*}
We choose as Gauss map
$$N=-\sin\theta E_2+\cos\theta E_3$$
and this will be the chosen orientation on $S$ throughout this paper.  Let $H$ and $K_{ext}$ be the mean curvature and the extrinsic Gaussian curvature of $S$, respectively. Using classical notation,
$$H=\frac12\frac{Eg-2Ff+Ge}{EG-F^2},\hspace*{1cm}K_{ext}=\frac{eg-f^2}{EG-F^2},$$
with
$$\begin{array}{lll}
 E=\langle e_1,e_1\rangle, &F=\langle e_1,e_2\rangle, &G=\langle e_2,e_2\rangle.\\
e=\langle N,\nabla_{e_1}e_1\rangle, & f=\langle N,\nabla_{e_1}e_2\rangle, &g=\langle N,\nabla_{e_2}e_2\rangle.
\end{array}$$
In our case, the coefficients of the first fundamental form are
$$E=1,\hspace*{.5cm} F=0,\hspace*{.5cm} G=e^{2z},$$
and $EG-F^2=e^{2z}$. The values of $\nabla_{e_i}e_j$ are
\begin{eqnarray*}
\nabla_{e_1}e_1&=&(\theta'+\cos\theta)(-\sin\theta E_2+\cos\theta E_3).\\
\nabla_{e_1}e_2&=&\nabla_{e_2}e_1=\sin\theta e^{z}E_1.\\
\nabla_{e_2}e_2&=&-e^{2z}E_3.
\end{eqnarray*}
 Then
$$e=\theta'+\cos\theta,\hspace*{.5cm} f=0,\hspace*{.5cm} g=-e^{2z}\cos\theta.$$
As conclusion,
\begin{equation}\label{curv}
H=\frac12\ \theta',\hspace*{1cm}K_{ext}=-\cos\theta(\theta'+\cos\theta).
\end{equation}
Hence that  the principal curvatures are
$$\kappa_1=\theta'+\cos\theta,\hspace*{1cm}\kappa_2=-\cos\theta.$$
In order to obtain the intrinsic Gauss curvature $K_{int}$, recall that
$K_{int}=K_{ext}+K(e_1\wedge e_2)$, where $K(e_1\wedge e_2)$ is the sectional curvature of each tangent plane and
$$K(e_1\wedge e_2)=\frac{\langle\nabla_{e_1}\nabla_{e_2} e_2-\nabla_{e_2}\nabla_{e_1} e_2-\nabla_{[e_1,e_2]}e_2,e_1\rangle}{EG-F^2}.$$
Now
\begin{eqnarray*}
\nabla_{e_1}\nabla_{e_2} e_2&=&\nabla_{e_1}(-e^{2z}E_3)=e^{2z}(\cos\theta E_2-2\sin\theta E_3).\\
\nabla_{e_2}\nabla_{e_1}e_2 &=& \nabla_{e_2}(\sin\theta e^{z} E_1)= -\sin\theta e^{2z}E_3.\\
\nabla_{[e_1,e_2]}e_2&=&0.\\
\end{eqnarray*}
Thus
$$K(e_1\wedge e_2)=\cos^2\theta-\sin^2\theta.$$
As consequence
\begin{equation}\label{kint}
K_{int}=-\theta'\cos\theta-\sin^2\theta.
\end{equation}
If we consider an invariant surface $S$ of Sol, any condition of its curvature writes as an ordinary differential equation ${\cal E}(s,\theta,\theta')=0$ on the function $\theta$. In order to calculate $S$, we have to obtain the generating curve $\alpha$, and so, we need to solve ${\cal E}=0$ together the system
\begin{eqnarray}
y'(s)&=&e^{z(s)}\cos\theta(s)\label{s-1} \\
z'(s)&=&\sin\theta(s)\label{s-2}.
\end{eqnarray}
We can also assume that $\alpha$ is locally a graph on the $y$-axis or on the $z$-axis. If $\alpha$ writes as
$\alpha(y)=(0,y,z(y))$, the change of variables is
\begin{equation}\label{change}
\theta'(s)\rightarrow e^{2z}\frac{z''+z'^2}{(1+z'^2 e^{2z})^{3/2}},\ \sin\theta\rightarrow\frac{z'e^z}{\sqrt{1+z'^2 e^{2z}}},\
\cos\theta\rightarrow \frac{1}{\sqrt{1+z'^2e^{2z}}}.
\end{equation}
Depending on each case, we use interchangeably (\ref{s-1})-(\ref{s-2}) by \ref{change}).

In what follows, we  omit the integration constants for the function $y(s)$, since it represents an isometry  of the surface by translations of type $T_{2,c}$. Similarly, we omit the additive constants of the function $z$: in this case, the isometry $\phi(x,y,z)=(e^\lambda x,e^{-\lambda}y,z-\lambda)$ converts the generating curve
$s\longmapsto (0,y(s),z(s)+\lambda)$ into $s\longmapsto (0,e^{-\lambda}y(s), z(s))$.

%%%%%%%%%%%%%%%%%%%%%%%%%%%%%%%%%%%%%%%%%%%%%%%%%%%%%%%%%%%%%%%%%%%%%%%%
\section{Surfaces with constant mean curvature}\label{sect3}
%%%%%%%%%%%%%%%%%%%%%%%%%%%%%%%%%%%%%%%%%%%%%%%%%%%%%%%%%%%%%%%%%%%%%%%%

\begin{theorem}\label{t-minimal}
The only invariant minimal  surfaces of Sol  are:
\begin{enumerate}
\item a leaf of the foliation ${\cal F}_2$ or,
\item a leaf of the foliation ${\cal F}_3$ or,
\item the surface generated by the graphic $z(y)=\log(y)$.
\end{enumerate}
\end{theorem}

\begin{proof}
If $S$ is minimal, then $\theta'=0$, that is, $\theta(s)=\theta_0$ for some constant $\theta_0\in\r$ and  $z(s)=(\sin\theta_0)s$. If $\sin\theta_0=0$, then (\ref{s-1})-(\ref{s-2}) gives $z(s)=\lambda$ and $y(s)=(\cos\theta_0)s$. This says that $\alpha$ is a horizontal straight-line and $S$ is a leaf of ${\cal F}_3$. Similarly, if $\cos\theta_0=0$, then $y$ is a constant function, $\alpha$ is a vertical straight-line and the surface
belongs to the family ${\cal F}_2$.

If $\sin\theta_0\not=0$, we have from (\ref{s-2}) that
$$y'=e^{ (\sin\theta_0)s}\cos\theta_0\Rightarrow y(s)=\cot\theta_0 e^{ (\sin\theta_0)s}.$$
Then
$$\alpha(s)=\Big(\cot\theta_0 e^{ (\sin\theta_0)s},(\sin\theta_0)s\Big).$$
This means that $\alpha$ describes the graphic of a logarithmic  function: $z=\log((\tan\theta_0)y)$.
\end{proof}
The above result can be obtained by using (\ref{change}): if the surface is not a graph of $z=z(y)$, then $y$ is a constant function and the surface is a leaf of ${\cal F}_2$. The minimality condition $H=0$ writes as $z''+z'^2=0$. Then $z$ is constant (and $S$ is a leaf of ${\cal F}_3$) or by integration, we have up constants that $\log(z')=-z$ and $e^z=y$.
\begin{theorem}
Let $S$ be an invariant surface in Sol with constant mean curvature $H\not=0$. We write $\alpha(s)=(0,y(s),z(s))$ the generating curve of $S$. Then
\begin{enumerate}
\item The curve $\alpha$ is invariant by a discrete group of translation in the $y$-direction.
\item The $z$-coordinate  is bounded and periodic.
\item The curve $\alpha$ has self-intersections.
\item The velocity vector of $\alpha$ turns around the origin.
\end{enumerate}

\end{theorem}

\begin{proof}
 From $\theta'=2H$, we know that, up constants,  $\theta(s)=2Hs$. Then $z'(s)=\sin(2Hs)$ and
$$z(s)=-\frac{1}{2H}\cos(2Hs).$$
In particular, $z$ is a periodic function of period $T=\pi/H$, whose derivative vanishes in a discrete set of points, namely,
$A=\{n\pi/(2H);n\in\z\}$.   From (\ref{s-1}),
$$y'=\exp{\Big(-\frac{1}{2H}\cos(2Hs)\Big)}\cos(2Hs).$$
Then the function $y'$ vanishes at $B=A+\pi/2$ and this means that $\alpha$ is not a graph on the $y$-axis, being the velocity of $\alpha$  vertical at each point of $B$. Moreover, $z$ takes the same value at these points:  with our choice of the integration constants, this value is $z=0$.

It is easy to show that if $\{y(s),z(s),\theta(s)\}$ satisfy (\ref{s-1})-(\ref{s-2}) and $\theta'=2H$, with
initial conditions $\{y_0,z_0,\theta_0\}$, then the functions
$\{y(s+T)-y(T)+y_0,z(s),\theta(s)\}$ satisfy the same equations and initial conditions. By uniqueness of solutions of O.D.E,  they must agree with the first set of solutions. In particular, $y(s+T)=y(s)+y(T)-y_0$. Thus, we have proved that the generating curve $\alpha$ is invariant by translations of the group of translations generated by the  vector $(0,y(T)-y_0,0)$. In our notations, this group is $\{T_{2,n(y(T)-y_0)};n\in\z\}$.

Finally, the function $\theta(s)$ takes all real values, which means that the planar velocity vector $\alpha'(s)=\cos\theta(s)E_2(s)+\sin\theta(s)E_3(s)$ goes taking all the values of a unit circle $\s^1$ in a monotonic sense.
\end{proof}

\begin{figure}[hbtp]\begin{center}
\includegraphics[width=5cm]{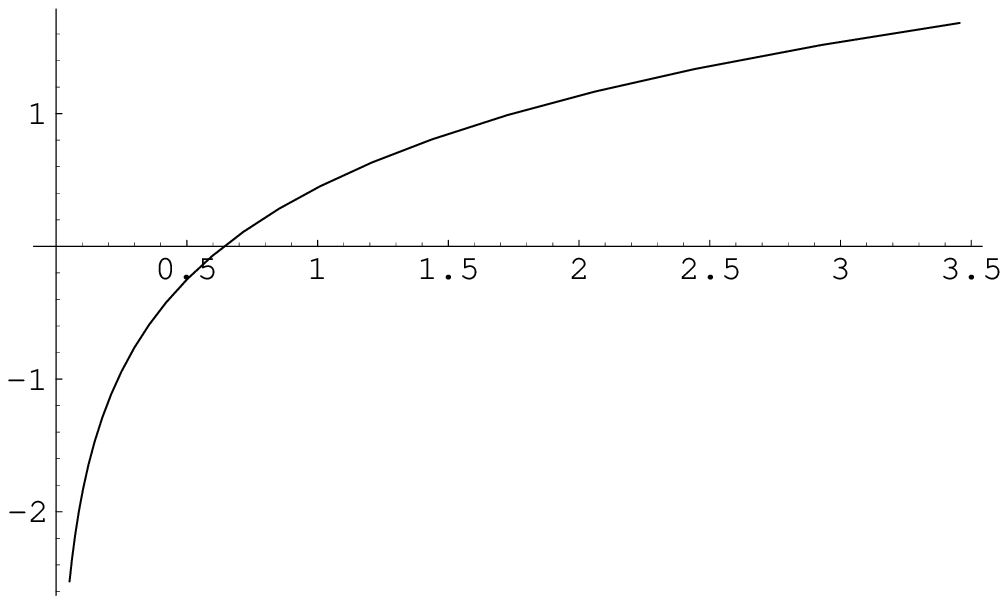}\hspace*{1cm}\includegraphics[width=5cm]{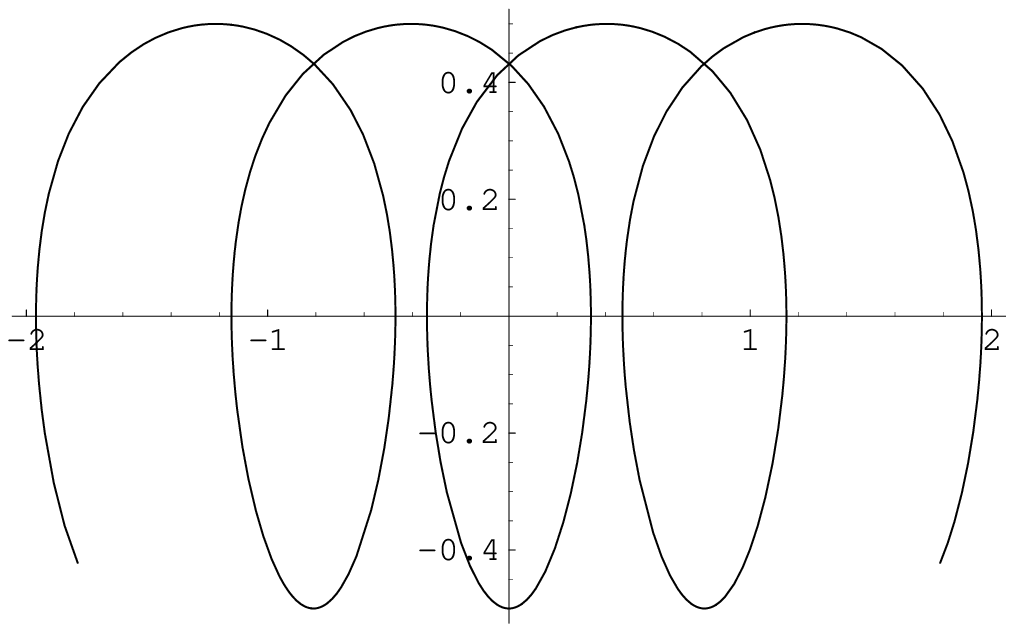}\end{center}
\caption{Generating curves of invariant surfaces with constant mean curvature:  case $H=0$ (left);
case $H=1$ (right).\label{sol-fig1}}
\end{figure}

%%%%%%%%%%%%%%%%%%%%%%%%%%%%%%%%%%%%%%%%%%%%%%%%%%%%%%%%%%%%%%%%%%%%%%%%
\section{Surfaces with constant Gaussian curvature}\label{sect4}
%%%%%%%%%%%%%%%%%%%%%%%%%%%%%%%%%%%%%%%%%%%%%%%%%%%%%%%%%%%%%%%%%%%%%%%%

In this section we study invariant surfaces in Sol with constant (intrinsic or extrinsic) Gaussian curvature.

\begin{theorem}\label{t-k1} Let $S$ be an invariant surface in Sol with constant intrinsic Gaussian curvature $K_{int}=c$. Up integration constants, we have the next classification:

\begin{enumerate}
\item If $c=0$, the surface is a leaf of ${\cal F}_3$ or the generating curve $\alpha$ of $S$ is
$$\alpha(s)=(0,\frac12 \Big(s\sqrt{s^2-1}-\log(s+\sqrt{s^2-1})\Big),\log(s)),\ s^2\geq 1.$$
\item If $c=-1$, the surface is a leaf of the foliation ${\cal F}_2$ or the generating curve $\alpha$ of $S$ is
the graph of $z(y)=\log(\cosh(y))$.
\item If $c\in (-1,0)$, then $\alpha$ is a graph of type $z(y)=\log(y)$, or $z(y)$ is defined in all the real line $\r$ with a single minimum, or $z(y)$ is a monotonic function defined in some interval $(a,\infty)$.
\item If $c>0$ or $c<-1$, $z(y)$ is a bounded function defined in a bounded interval $(a,b)$ with a single maximum or minimum and it is  vertical at the end points of $(a,b)$.
\end{enumerate}
Moreover, except that $S$ is a leaf of ${\cal F}_2$, the generating curve is a graph of a function $z=z(y)$.
\end{theorem}

\begin{proof}

Equation (\ref{kint}) writes as $\theta'\cos\theta+\sin^2\theta=-c$ or
\begin{equation}\label{kint2}
(\sin\theta)'+\sin^2\theta+c=0.
\end{equation}
If $\cos\theta=0$ at some point, then
$c=-1$. In this case, if $\cos\theta\equiv 0$, then $y$ is a constant function. This means that $\alpha$ is a vertical straight-line and $S$ is a leaf of ${\cal F}_2$. Thus, if $c\not=-1$, $\cos\theta\not=0$ and from (\ref{s-1}), $\alpha$ is the graph of $z=z(y)$.

If we put $p=\sin\theta$, then (\ref{kint2}) writes $p'+p^2+c=0$, that is,
\begin{equation}\label{kint3}
\frac{p'}{p^2+c}=-1.
\end{equation}
This equation makes sense only if $p^2+c\not=0$ and then, it can be integrated by parts.
 On the contrary, that is,
if $\sin^2\theta+c=0$, then $c\in[-1,0]$. We distinguish cases:
\begin{enumerate}
\item Case $c=-1$. We know that $S$ is a leaf of ${\cal F}_2$ if $\cos\theta\equiv 0$. On the contrary, a first integration  gives $\sin\theta=\tanh(s)$ and from (\ref{s-2}), $z(s)=\log(\cosh(s))$. Then $y'(s)=1$, that is, $y(s)=s$. This means that $\alpha$ is the graphic of $z(y)=\log(\cosh(y))$.
\item Case $c=0$.  If $\sin\theta\equiv 0$, then (\ref{s-2}) shows that $z$ is a constant function,
 $\alpha$ is a horizontal curve and $S$ is a leaf of ${\cal F}_3$. On the other case, $\sin\theta=1/s$ and by (\ref{s-2}), we have $z(s)=\log(s)$. It is possible to solve (\ref{s-1}) obtaining
$$y(s)=\frac12 \Big(s\sqrt{s^2-1}-\log(s+\sqrt{s^2-1})\Big).$$
\item Case $c\in (-1,0)$. If $c+\sin^2\theta=0$ at some point (for example, at $s=0$), the solution of (\ref{s-1})-(\ref{s-2}) is up constants
$\theta(s)=\theta(0):=\theta_0$,
 $y(s)=\cot\theta_0e^{(\sin\theta_0) s}$ and $z(s)=(\sin\theta_0) s$. This means that
 $\alpha$ is the graphic of $z(y)=\log((\tan\theta_0) y)$. Finally, we assume that $\sin^2\theta+c\not=0$ at some point (for example, at $s=0$).   A first integration of (\ref{kint3}) depends on the sign of $\sin^2\theta_0+c$.
\begin{enumerate}
\item Assume $\sin^2\theta+c<0$. Then (\ref{kint3}) gives $\sin\theta=\sqrt{-c}\tanh(\sqrt{-c}\ (s+\lambda))$. Letting $s\rightarrow\infty$, we conclude that $\sin\theta$ vanishes at some point. Without loss of generality, we suppose that this occurs at $s=0$. Moreover, $z'$ vanishes only at one point, namely, $s=0$, and $z''(s)>0$. This means that $z=z(s)$ is a convex function with only a single minimum at $s=0$. Finally,
$$|y'(s)|=\cosh(\sqrt{-c}s)\sqrt{1+c\tanh^2(\sqrt{-c}s)}\geq  \sqrt{1+c}$$
which means that the function $y$ is defined in all $\r$. Thus $z=z(y)$ with $y\in\r$. Since $z'(y)=z'(s)/y'(s)$, we know that $y=0$ is the only extremum of $z(y)$ and from (\ref{change}), we conclude $z''(0)=-c>0$, that is, $z=z(y)$ has a minimum at $y=0$.
\item Assume $\sin^2\theta+c>0$. Now (\ref{kint3}) gives $\sin\theta=\sqrt{-c}\cot(\sqrt{-c}(s+\lambda))$, which is defined in an open interval of $\r$ of type $(a,\infty)$. Suppose that $\lambda$ is chosen to the domain of $\sin\theta$ is $(0,\infty)$, that is, we take $\sin\theta=\sqrt{-c}\cot(\sqrt{-c}s)$. Then $z$ is an increasing function and
$z(s)=\log(\sinh(\sqrt{-c}s))$. Moreover,
$$|y'(s)|=\sinh(\sqrt{-c}s)\sqrt{1+c\cot^2(\sqrt{-c}s)}\geq 1.$$
Thus $y$ is defined in an interval of type $(M,\infty)$.
\item Case $c>0$. Now $\sin\theta=-\sqrt{c}\tan(\sqrt{c}\ s)$ and $z(s)=\log(\cos(\sqrt{c}\ s))$. Moreover $z'$ vanishes at exactly one  point ($s=0$). The same reasoning as above shows that $z=z(y)$ has a maximum at that point. Since
$-1\leq\sin\theta\leq 1$, the values of $s$ lies in some bounded domain $I=(-M,M)$. The values of $y'(s)$ are bounded because
$$|y'(s)|\leq \cos(\sqrt{c}s)\sqrt{1-c\tan^2(\sqrt{c}s)}\leq \cos(\sqrt{c}s)\leq \cos(\sqrt{c}M).$$
Then the function $y$ takes values in some bounded domain $(-y_M,y_M)$. Finally,
$$\lim_{s\rightarrow \pm M}|z'(s)|=1,\hspace*{1cm}\lim_{s\rightarrow \pm M}|y'(s)|=0,$$
and this means that $\alpha$ is vertical at the points $\pm y_M$.
\end{enumerate}
\item Case $c<-1$.  Now the reasoning is similar than the case $c>0$. We have $\theta=\sqrt{-c}\tanh(\sqrt{-c}s)$ and
$z(s)=\log(\cosh(\sqrt{-c}s))$. The function $z=z(y)$ is convex with a minimum at the origin. Also, the function $z=z(y)$ is defined in some bounded domain $(-y_M,y_M)$ and the generating curve $\alpha$ is vertical at $\pm y_M$.

\end{enumerate}
\end{proof}

\begin{remark} \label{re-1}In the cases $c<-1$ and $c>0$ the derivatives of the functions $y(s)$ and $z(s)$ are bounded at the end points  of the maximal domain $(-M,M)$. However one can not continue the solutions because $\cos\theta\rightarrow 0$ and $\sin^2\theta\rightarrow 1$ as $s\rightarrow\pm M$ and so, from
 (\ref{kint2}), the function $\theta'$ goes to $\infty$ as $s\rightarrow\pm M$.
\end{remark}

\begin{figure}[hbtp]\begin{center}
\includegraphics[width=5cm]{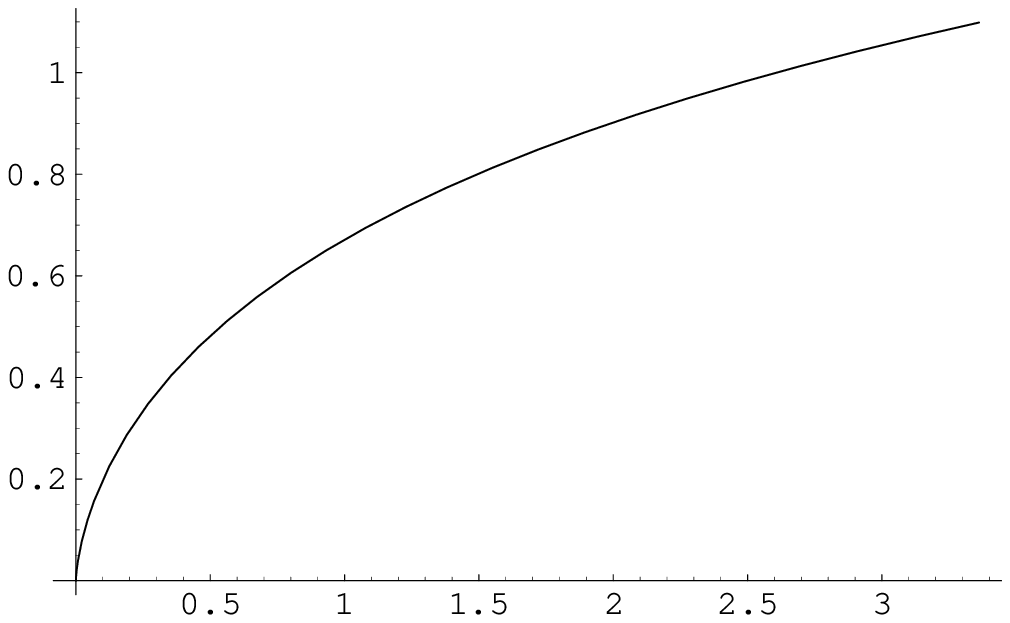}\hspace*{1cm}\includegraphics[width=5cm]{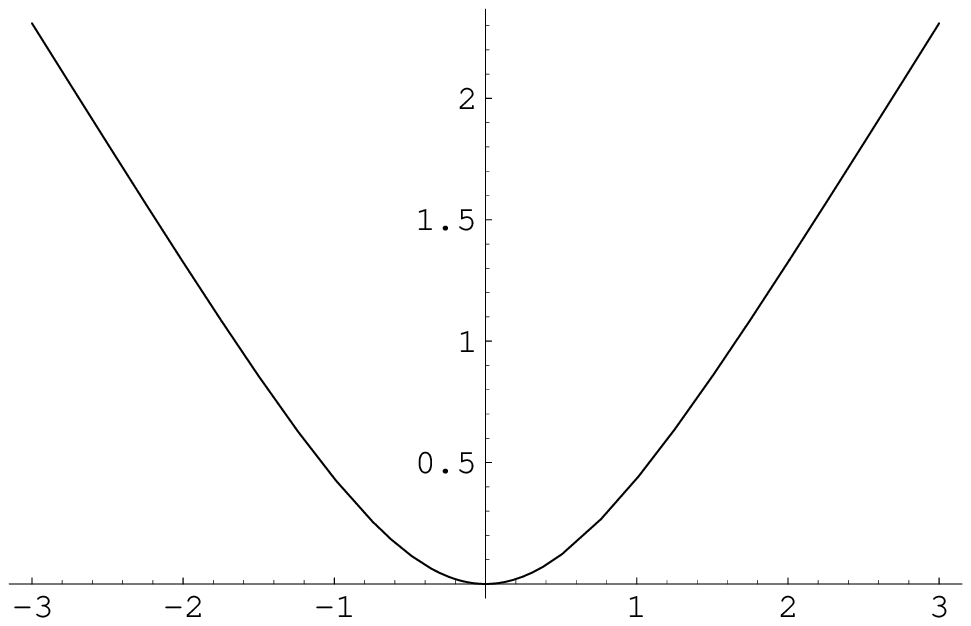}\\
\hspace*{1cm}\\
\includegraphics[width=5cm]{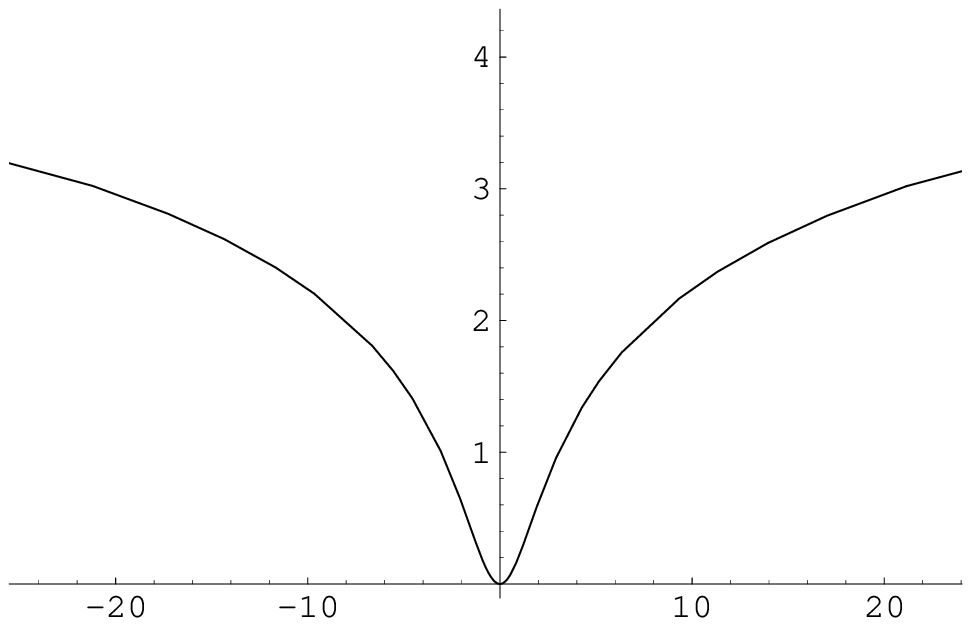}\hspace*{1cm}\includegraphics[width=5cm]{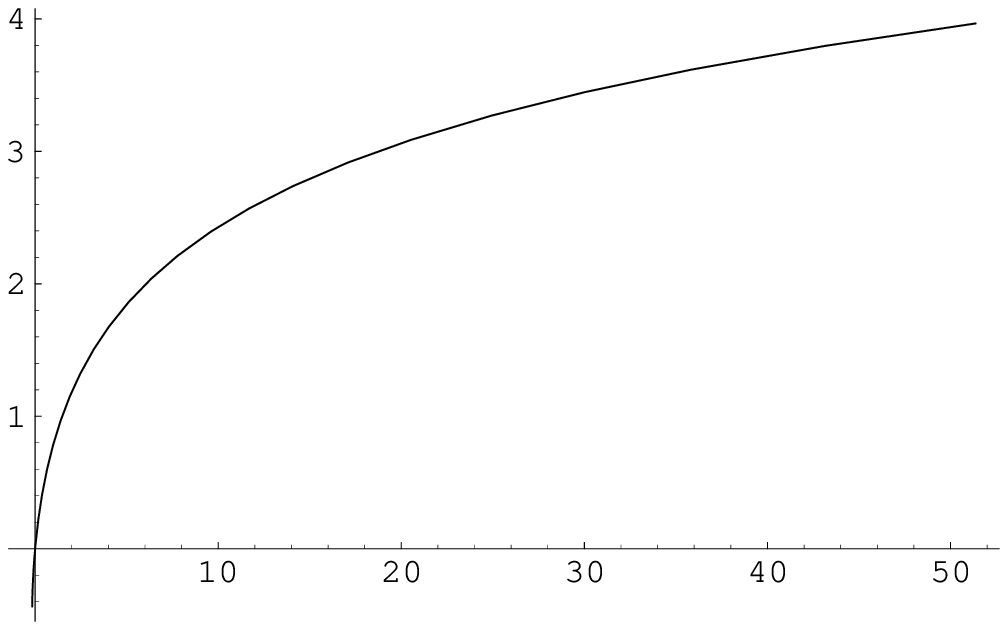}\end{center}
\caption{Generating curves of invariant surfaces with $K_{int}=0$ (top, left); $K_{int}=-1$ (top, right); $K_{int}=c\in(-1,0)$,  case $\sin\theta_0+c<0$ (bottom, left); and $K_{int}=c\in(-1,0)$, case $\sin\theta_0+c>0$ (bottom, right).\label{sol-fig2}}
\end{figure}

\begin{theorem} Let $S$ be an invariant surface in Sol with constant extrinsic Gaussian curvature $K_{ext}=c$. Up integration constants, we have the next classification:

\begin{enumerate}
\item If $c=0$, the surface is a leaf of ${\cal F}_2$ or  the generating curve $\alpha$ of $S$ is
$$\alpha(s)=(0,\tanh(s),-\log(\cosh(s))).$$

\item If $c=-1$, the surface is a leaf of ${\cal F}_3$  or
$$\alpha(s)=(0,-\frac{s^2-1}{s}+\log(s+\sqrt{s^2-1}),-\log(s)).$$
\item If $c\in (-1,0)$, then $\alpha$ is the graph $z(y)=\log(y)$; or $z(y)$ is defined in a bounded interval $(-M,M)\subset \r$ and it is asymptotic to vertical lines $y=\pm M$; or $z(y)$ is defined in a bounded interval $(m,M)$ being asymptotic to the vertical line $y=m$.
\item If $c>0$ or $c<-1$, the function $z(y)$ is defined in a bounded interval $I=(a,b)$ with a single maximum or minimum, it is bounded and it is  vertical at the end points of $I$.
\end{enumerate}
\end{theorem}

\begin{proof} From (\ref{curv}), $-\cos\theta(\theta'+\cos\theta)=c$,
 and as in the proof of Theorem \ref{t-k1}, put
$p=\sin\theta$. Now we have
$$\frac{p'}{p^2-c-1}=1.$$
The reasoning is similar as in Theorem \ref{t-k1}. First, we consider that $\sin^2\theta-c-1\equiv 0$, that is, $c\in [-1,0]$.  Then the sectional curvature $K(e_1\wedge e_2)$ is $-1-2c$ and $K_{int}=-1-c$. These cases were previously studied in Theorem \ref{t-k1} corresponding there with
$p^2+K_{int}=0$. In particular, if $c\in (-1,0)$ the solution is $z(y)=\log(y)$; if $c=0$, then $S$ is a leaf of ${\cal F}_2$ or $z(y)=\log(\cosh(y))$ and if $c=-1$, $S$ is a leaf of ${\cal F}_3$. The rest of cases are the following:
\begin{enumerate}
\item Case $c=-1$. Then $z(s)=-\log(s)$ and
$$y(s)=-\frac{s^2-1}{s}+\log(s+\sqrt{s^2-1}).$$
\item Case $c=0$. Then $z(s)=-\log(\cosh(s))$ and $y(s)=\tanh(s)$.
\item Case $c\in (-1,0)$. The function $\theta$ is given by $\sin\theta=-\sqrt{c+1}\tanh(\sqrt{c+1}s)$.
\begin{enumerate}
\item If $\sin^2\theta_0-c-1=0$, then the solution is, up constants, $z(y)=\log(y)$.
\item If $\sin^2\theta_0-c-1<0$, $z(s)=-\log(\cosh(\sqrt{c+1}s))$ and it is defined in the whole of $\r$. Again, $z(s)$ has a maximum at $s=0$. Now
$$|y(\infty)|=|y(0)|+\int_0^{\infty}|y'(t)|dt\leq |y(0)|+\int_0^{\infty}2e^{-\sqrt{(c+1)}t}dt<\infty.$$
This shows that the function $y$ takes values in some bounded interval $(-M,M)$. Thus the generating curve $z=z(y)$ is also defined in some bounded interval and since $z(s)$ takes values arbitrary big, the graphic of $\alpha$ is asymptotic to the two vertical lines   $y=\pm M$.
\item If $\sin^2\theta_0-c-1>0$, we obtain  $\sin\theta=-\sqrt{c+1}\coth(\sqrt{c+1}(s+\lambda))$, $\lambda\in\r$,  and
$z(s)=-\log(\sinh(-\sqrt{c+1}(s+\lambda))$. Assuming for example that $\sin\theta_0>0$,
the constant $\lambda$ is negative with
 $\sin\theta_0=-\sqrt{c+1}\coth(\sqrt{c+1}\lambda)$.  The function $z$ is monotonic on $s$ and it is  defined in some interval of type $(-\infty,M)$, where $1=(c+1)\coth^2(\sqrt{c+1}M)$. As
$$y'(s)=\frac{1}{\sinh(-\sqrt{c+1}(s+\lambda))}\sqrt{1-(c+1)\coth^2(\sqrt{c+1}(s+\lambda))},$$
the value of $y'(M)$ is bounded and
$$|y(-\infty)|<|y(0)|+\int_{-\infty}^0|y'(s)|ds\leq |y(0)|+\int_{-\infty}^0 \frac{1}{\sinh(-\sqrt{c+1}(s+\lambda))}<\infty.$$
This shows that the value of $y$ belongs an interval of type $(m,M)$. Thus $\alpha$ is asymptotic to the vertical line $y=m$.
\end{enumerate}
\item Case $c>0$. Now $\sin\theta=-\sqrt{c+1}\tanh(\sqrt{c+1}\ s)$ and $z(s)=-\log(\cosh(\sqrt{c+1}\ s))$.
The curve has a single maximum at $s=0$. From the expression of $\sin\theta$ and since $\sqrt{c+1}>1$, the variable $s$ can not take arbitrary values: exactly, $\theta$ is defined whenever  $(c+1)\tanh^2(\sqrt{c+1} s)\leq 1$. Then $\theta$ is defined in some bounded domain $(-M,M)$. On the other hand, and because
$$y'(s)=\frac{1}{\cosh(\sqrt{c+1}s)}\sqrt{1-(c+1)\tanh^2(\sqrt{c+1}s)},$$
the values $y'(\pm M)$ vanish and since the domain  of $s$ is bounded, $y$ takes values in some interval $(-y_M,y_M)$. Because $y'(\pm M)=0$ and $z'(\pm M)=1$, we conclude that the generating curve $\alpha$ is vertical at the points $\pm y_M$. Finally it is evident that the function $z(y)$ is bounded in the maximal domain.
\item Case $c<-1$. Then $\sin\theta=\sqrt{-c-1}\tan(\sqrt{-c-1}\ s)$ and $z(s)=-\log(\cos(\sqrt{-c-1}\ s))$. The function $z$ has a single minimum at $s=0$. Now the conclusions are similar as the case $c>0$, and we omit the details.
\end{enumerate}
\end{proof}
For the cases $c<-1$ and $c>0$ we can apply the same comments as in Remark \ref{re-1}.

%%%%%%%%%%%%%%%%%%%%%%%%%%%%%%%%%%%%%%%%%%%
\section{Linear Weingarten surfaces}\label{sect5}
%%%%%%%%%%%%%%%%%%%%%%%%%%%%%%%%%%%%%%%%%%%%%%%%

A generalization of umbilical surfaces, as well as, surfaces with constant mean curvature, are the Weingarten surfaces. A Weingarten surface is a surface that satisfies a smooth relation of type $W(\kappa_1,\kappa_2)=0$, where $\kappa_i$ are the principal curvatures of the surface. Equation $W(\kappa_1,\kappa_2)=0$ gives other relation of type $U(H,K_{ext})=0$. Among the choices of $W$ and $U$, the simplest case is that they are linear on its variables. So, we say that $S$ is a {\it linear Weingarten surface} if  satisfies one of the two (non-equivalent) conditions:
\begin{equation}\label{w1}
a\kappa_1+b\kappa_2=c,
\end{equation}
or
\begin{equation}\label{w2}
aH+bK_{ext}=c,
\end{equation}
where $a$, $b$ and $c$ are constant. In particular, if $a=-b$, $c=0$ in (\ref{w1}) we have umbilical surfaces, whereas if $a=b$, the surface has constan mean curvature. In (\ref{w2}), the choices $b=0$ and $a=0$ give surfaces with constant mean curvature or constant extrinsic Gauss curvature, respectively. In terms of the angle function $\theta$, equations (\ref{w1}) and (\ref{w2}) write as $a\theta'+(a-b)\cos\theta=c$  and $(a-2b\cos\theta)\theta'-2b\cos^2\theta=2c$,   respectively.

A complete study of the solutions of above two equations is not difficult, although the number of cases depending on the constants $a$, $b$ and $c$ makes lengthy the statements of results. For example, a simple case is the choice $a=0$ in (\ref{w1}): the generating curve $\alpha$ satisfies that  $\cos\theta$ is a constant function, that is, $\theta$ is a constant function $\theta_0$.
Then the generating curve is $\alpha(s)=(0,(\cot\theta_0)e^{(\sin\theta_0)s},(\sin\theta_0)s)$.

In order to simplify the proofs, we  are going to consider in this section the linear relation  (\ref{w1})
when $c=0$. So we will assume that $\kappa_1=m\kappa_2$.

\begin{theorem}\label{t-w} Let $S$ be an invariant surface in Sol that satisfies a relation of type
$\kappa_1=m\kappa_2$. Then $S$ is a leaf of ${\cal F}_2$ or we have the following classification according to the values of the parameter $m$:
\begin{enumerate}
\item If $m=1$ the surface is an umbilical surface.
\item If $m=-1$, the surface is a minimal surface.
\item If $m>-1$ or $m<-2$, then the generating curve $\alpha$ is a graph of $z=z(y)$, with a single minimum ($m<-2$) or single maximum ($m>-1$). Moreover, $\alpha$ is asymptotic to two vertical lines.
\item If $m\in (-2,-1)$, $\alpha$ is a graph of $z=z(y)$ defined in the whole of $\r$ and it presents a single minimum.
\item If $m=-2$, $\alpha$ is given by the graph of $z(y)=\log(\cosh(y))$.
\end{enumerate}
\end{theorem}
\begin{proof}
The generating curve $\alpha$ is given by
\begin{equation}\label{w11}
\theta'+(1+m)\cos\theta =0.
\end{equation}
We discard the case $m=1$ that gives umbilical surfaces, which were studied in  \cite[Proposition 19]{sot}, and the case $m=-1$, which corresponds with the minimal case studied in Theorem \ref{t-minimal}. If $\theta'$ vanishes at some point $s_0$, then $\cos\theta(s_0)=0$. By uniqueness of solutions, $\theta(s)=\pm\pi/2$, that is, $\theta$ is a constant function. Moreover, $z(s)=\pm s$ and $y(s)$ is a constant function. This means  that $\alpha$ is a vertical straight-line and $S$ is a leaf of ${\cal F}_2$. In other words, each leaf of ${\cal F}_2$ satisfies the relation $\kappa_1=m \kappa_2$ for any $m$, since $\kappa_1=\kappa_2=0$ on $S$.

On the contrary, we assume that  $\theta'\not=0$. Up constants, an integration of (\ref{w11}) gives
     $$\theta(s)=-2\arctan(\tanh(\frac{m+1}{2}s)),$$
that is, $\theta(0)=0$.
Taking limits, we obtain
     $$\lim_{s\rightarrow\pm\infty}\theta(s)=\mp\frac{\pi}{2}.$$
     As $y'=e^z\cos\theta$, this means that $y'\not=0$ and  $\alpha$ is a graph  $z=z(y)$.     Now, we have
     $$\sin\theta(s)=-\tanh((m+1)s),\hspace*{.5cm} \cos\theta(s)=\frac{1}{\cosh((m+1)s)}$$
     $$z(s)=-\frac{1}{m+1}\log(\cosh((m+1)s)),\hspace*{.5cm}y'(s)=\Big(\cosh((m+1)s)\Big)^{-\frac{m+2}{m+1}}.$$
We distinguish cases.
\begin{enumerate}
\item Assume $\frac{m+2}{m+1}>0$, that is, $m>-1$ or $m<-2$, then
$$y(\infty)-y(0)\leq \int_0^{\infty}|y'(s)|ds\leq \frac{1}{|m+1|}\int_0^{\infty}(e^{-t})^\frac{m+2}{m+1}dt<\infty.$$

 This shows that the the function $y$ is bounded. Then the function $z(y)$ is defined in a bounded domain $I=(-M,M)$. As $z(\pm\infty)=\infty$, this means that the graphic of the generating curve $\alpha$ is asymptotic to two vertical lines at $y=\pm M$. On the other hand, $z'$ only vanishes at $s=0$  and $z''(s)=\theta'\cos\theta=-(m+1)\cos^2\theta$. This implies that $z$ (or $\alpha$) has an absolute minimum or absolute maximum depending if $m+1<0$ or $m+1>0$, respectively.
\item Case that $m\in(-2,-1)$. The function $z(s)$ takes arbitrary values with a minimum at $s=0$. On the other hand, $y'(s)\geq 1$, and so $y=y(s)$ takes values in all $\r$. Thus the generating curve $\alpha$ is a graph of the function $z=z(y)$ defined for any $y\in\r$.
\item Case $m=-2$. We find that $z(s)=\log(\cosh(s))$ and $y(s)=s$. Thus $\alpha$ is the graph of $z(y)=\log(\cosh(y))$.
\end{enumerate}

\end{proof}

\begin{remark} If we put $m=1$, the linear Weingarten says that $S$ is an umbilical surface. As we have pointed out, umbilical surfaces have been studied in \cite{sot}. The analytic properties obtained there agree with our  results corresponding to the more general case $m>-1$ in Theorem \ref{t-w}.
\end{remark}

\begin{figure}[hbtp]\begin{center}
\includegraphics[width=5cm]{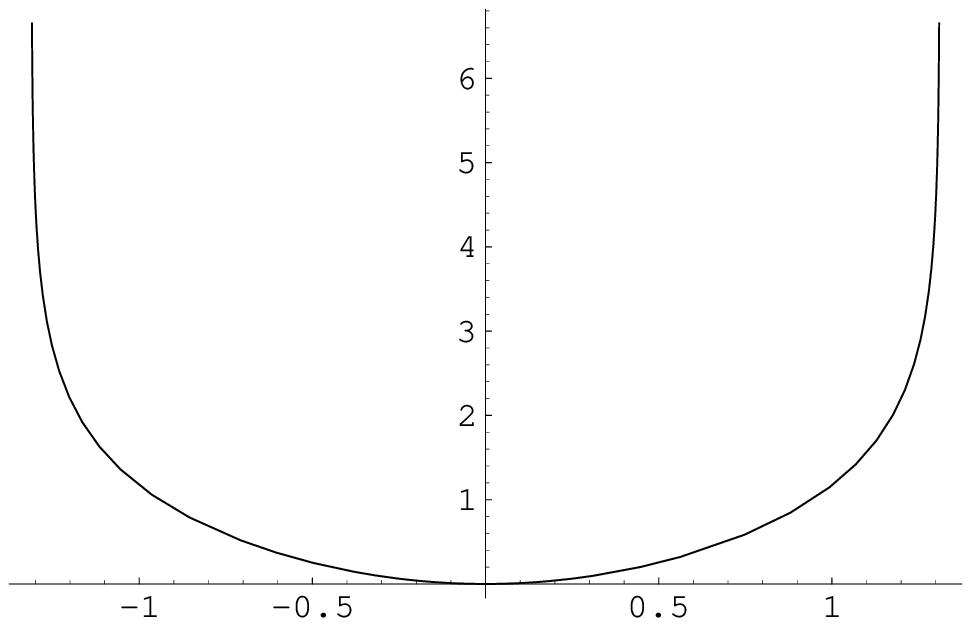}\hspace*{1cm}\includegraphics[width=5cm]{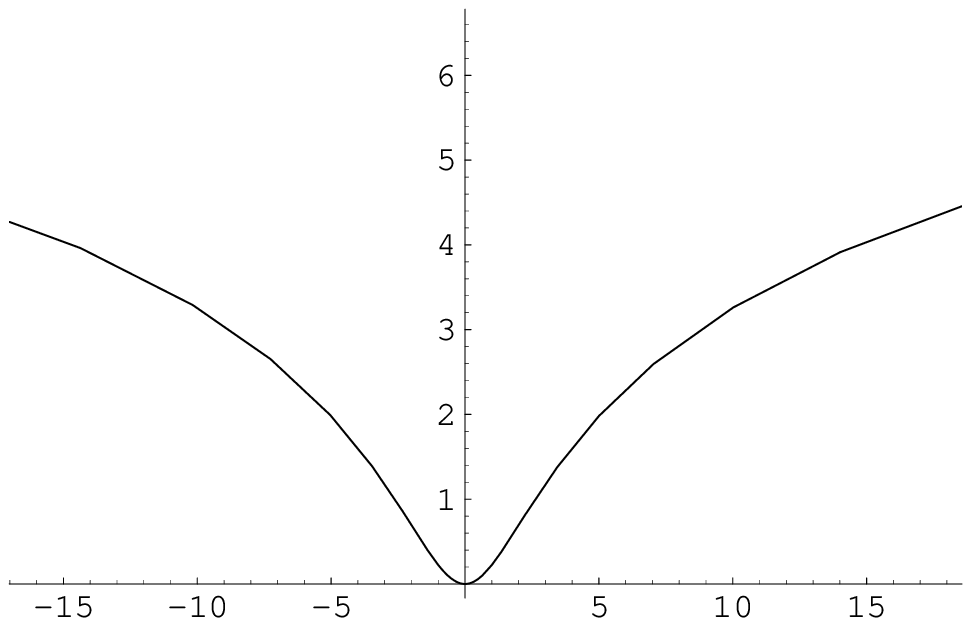}\end{center}
\caption{Linear Weingarten surfaces with $\kappa_1=m\kappa_2$: $m=-3$  (left) and $m=-3/2$ (right).\label{sol-fig3}}
\end{figure}

%%%%%%%%%%%%%%%%%%%%%%%%%%%%%%%%%%%%%%

\end{document}